\documentclass[12pt,reqno]{amsart}
\usepackage[utf8]{inputenc}
\usepackage[T1]{fontenc}
\usepackage[usenames, dvipsnames]{color}
\usepackage{ulem}

\usepackage{dsfont, amsfonts, amsmath, amssymb,amscd, stmaryrd, latexsym, amsthm, dsfont}
\usepackage[frenchb,english]{babel}
\usepackage{enumerate}
\usepackage{longtable}
\usepackage{geometry}
\usepackage{float}
\usepackage{tikz}
\usetikzlibrary{shapes,arrows}
\usepackage{float}
\usepackage{tikz}
\usepackage{xypic}
\usetikzlibrary{shapes,arrows}

\theoremstyle{plain}

\newtheorem{theorem}{Theorem}[section]
\newtheorem{lemma}[theorem]{Lemma}
\newtheorem{proposition}[theorem]{Proposition}

\theoremstyle{remark}

\def\QQ{\mathbb{Q}}

\usepackage{hyperref}
\hypersetup{
	colorlinks=true,
	urlcolor=blue,
	citecolor=blue}

\begin{document}

\selectlanguage{english}
\title[On Hilbert genus fields ...]{On Hilbert genus fields of   imaginary cyclic quartic  fields}

  \author[M. A.  Hajjami]{Moulay Ahmed Hajjami} 
  \address{Moulay Ahmed Hajjami, Department of  Mathematics, Faculty of  Sciences and Technology, Moulay Ismail University of Meknes, Errachidia, Morocco.}
  \email{a.hajjami76@gmail.com}

  \author[M. M. Chems-Eddin]{Mohamed Mahmoud Chems-Eddin}
  \address{Mohamed Mahmoud Chems-Eddin: Mohammed First University, Mathematics Department, Sciences Faculty, Oujda, Morocco }
  \email{2m.chemseddin@gmail.com}


\begin{abstract}
 Let $p$ be a prime number such that $p=2$ or $p\equiv 1\pmod 4$. Let $\varepsilon_p$ denote the  fundamental unit of  $\mathbb{Q}(\sqrt{p})$ and let $a$ be a positive square-free integer. The main aim of this paper is to determine explicitly the Hilbert genus field of  the imaginary cyclic quartic    fields of the form $\mathbb{Q}(\sqrt{-a\varepsilon_p\sqrt{p}})$.
\end{abstract}

\selectlanguage{english}

\maketitle
  \section{\bf Introduction}\label{sec:1}
 The study of the unramified extensions of a given number field $k$ is of a huge interest in algebraic number theory. For instance, the importance of the Hilbert class field of $k$, denoted by $H(k)$,  is represented in the fact that it is
 the maximal abelian unramified extension of $k$ and its Galois group    over  $k$, i.e., $\mathrm{G}:=\mathrm{Gal}(H(k)/k)$,  is isomorphic to $\mathbf{C}l(k)$, the class group of $k$ (cf. \cite[p. 228]{ref12}). Another important example of these unramified extensions  is  the genus field of $k$, which is defined as  the maximal extension of $k$ which is unramified 
 at all finite and infinite primes of $k$ of the form $kk_1$, where $k_1$ is an abelian extension of $\mathbb{Q}$ (cf. \cite{ref11}). These two fields have been largely investigated in old and recent studies (e.g. \cite{ref1,taous2008,chemszekhniniaziziUnits1,chemsazizizekhninijerr,cohn,kisilvsky}).
 
 Another very interesting  example  of  unramified extensions of $k$ is  the Hilbert genus field of $k$ which  is the invariant field $E(k)$ of $\mathrm{G}^2$. Then, by Galois theory, we  have:
 $$\mathbf{C}l(k)/\mathbf{C}l(k)^2  \simeq \mathrm{G}/\mathrm{G}^2 \simeq \mathrm{Gal}(E(k)/k),$$
 and thus, $2$-rank $(\mathbf{C}l(k))$ = $2$-rank $(\mathrm{Gal}(E(k)/k))$. On the other hand, $E(k)/k$ is the maximal unramified Kummer extension of exponent $2$. Thus, by Kummer theory (cf. \cite[p. 14]{ref16}), there exists a unique multiplicative group $\Delta$ such that
 $$
 E(k)=H(k)\cap k(\sqrt{k^*})= k(\sqrt{\Delta})  \text{  and }  {k^*}^2 \subset {\Delta} \subset k^*.
 $$
 Therefore, the question that arises is how to construct  the Hilbert genus field of $k$, or equivalently, how to give a set of generators for the finite group $\Delta/k{^*}^2$.
 Note that many    mathematicians have  investigated this question for some   biquadratic number fields. For example,   Bae and Yue   studied the Hilbert genus field of the fields $\mathbb{Q}(\sqrt{p}, \sqrt{d})$, for 
 a prime number $p$ such that, $p=2$ or $p\equiv 1\pmod 4$, and    a positive square-free  integer $d$ (cf. \cite{ref2}).

 Recently, Ouang and Zhang have determined the Hilbert genus field of the imaginary biquadratic fields $ \mathbb{Q}(\sqrt{\delta}, \sqrt{d})$, where $\delta =-1, -2$  or $-p$ with $p\equiv 3\pmod 4$ a prime number and $d$ any square-free integer. Thereafter,  they   constructed   the Hilbert genus field of real biquadratic fields $\mathbb{Q}(\sqrt{\delta}, \sqrt{d})$, 
 for   any  positive square-free integer $d$, and   $\delta=p, 2p$ or $p_1p_2$ where $p$, $p_1$ and $p_2$ are prime numbers congruent to $3\pmod 4$, such that the class number of $ \mathbb{Q}(\sqrt{\delta})$ is odd (cf.  \cite{ref17,ref18}). For more works on this problem, we refer the reader to the  papers \cite{ref20,ref19,ref7}.

 In the present work, using  other easier   techniques based on genus fields, we shall construct the Hilbert genus fields of  imaginary cyclic quartic fields of the form  $K=\mathbb{Q}(\sqrt{-a{\varepsilon}_p \sqrt{p}})$, for a prime number $p$ such that $p=2$ or $p\equiv 1\pmod 4$, and a positive square-free integer $a$ relatively prime to $p$.
 
 The plan of this paper is the following. In Section \ref{sec:2}, we shall collect some results of cyclic quartic fields theory and genus theory. Section \ref{sec:3} is 
 dedicated to the investigation of the genus field of the imaginary cyclic quartic fields $K$. In Section \ref{sec:4}, we will construct the Hilbert genus fields of the fields $K$. Therein,  we   give some numerical examples.

 \begin{center}
 	{	\large \bf{Notations}}
 \end{center}
 Let $k$ be a number field. The next notations will be used for the rest of
 this article:
 \begin{enumerate}[$\bullet$]
 	\item $\mathcal{O}_k$:  the ring of integers of $k$,
 	\item $\mathbf{C}l(k)$: the class group of $k$,
 	\item $h(k)$: the class number of $k$, 
 	\item $N_{k/k'}$:  the norm map of an extension $k/k'$,
 	\item $E_k$: the unit group of $k$,
 	\item $k^*$: the nonzero elements of  $k$,
 	\item $k^{(*)}$: the   genus field of  $k$,
 	\item $E(k) $:  the Hilbert genus field of $k$,
 	\item $H(k) $:  the Hilbert  class field of $k$,
 	\item $\delta_k$: the absolute discriminant of $k$,
 	\item $\delta_{k/k'}$: the generator of the relative discriminant of an extension $k/k'$,
 	\item $r_{2}(A)$:  the $2$-rank of a finite  abelian group $A$,
 	\item  $\varepsilon_{d} $:  the fundamental unit of $\QQ(\sqrt{d})$, where $d$ is a positive square-free integer,
 	\item $p$: a prime number such that $p=2$ or $p\equiv 1\pmod 4$,
 	\item $a$: a positive square-free  integer relatively prime to $p$,
 	\item $\delta=-a\varepsilon_p \sqrt{p}$,
 	\item $k_0=\mathbb{Q}(\sqrt{p})$,
 	\item $K=k_0(\sqrt \delta)$: an imaginary quartic cyclic number field,
 	\item $q_j$: an odd prime integer,
 	\item $\left(\dfrac{\cdot}{\cdot}\right)$:  the Legendre symbol,	
 	\item $\left(\frac{ -, -}{\mathcal{P}_{i}}\right)$:  the Hilbert symbol over $k_0$.           
 \end{enumerate}
 For more notations see   the beginning of each section below.  
 \section{\bf Preliminary results}\label{sec:2}
 
 In this section, we start by recalling some results that we will need in what follows. 
 Let $L$ be a cyclic quartic extension of the rational number field $\mathbb{Q}$. It is known   that  $L$ can be expressed uniquely  in the form: 
 $$L=\mathbb{Q}(\sqrt{a(d+b\sqrt{d})}),$$ 
 for some integers $a$, $b$, $c$ and $d$ such that  $d=b^2+c^2$ is square-free with $b>0$ and $c>0$, and
 $a$ is an odd square-free integer relatively prime to $d$ (cf. \cite{ref6,ref21}).  Note that $L$ possesses a unique quadratic subfield $k=\mathbb{Q}(\sqrt{d})$.

 \begin{lemma}[\cite{ref9}]\label{lemma 2.1; discriminants}
 	Keep the above notations. We have:
 	\begin{enumerate}[\rm 1.]
 		\item  	The absolute discriminant of $L$ is given by $\delta_L$, where:
 		$$
 		\delta_L=\left\{\begin{array}{lll}
 			2^8a^2d^3, & \text { if } & d \equiv 0\pmod 2, \\
 			2^4a^2d^3, & \text { if } & d \equiv 1\pmod 4, b \equiv 0\pmod 2, a+b \equiv 3\pmod 4, \\ 
 			a^2d^3,& \text { if } & d \equiv 1\pmod 4, b \equiv 0\pmod 2, a+b \equiv 1\pmod 4, \\ 
 			2^6a^2d^3, & \text { if } & d \equiv 1\pmod 4, b \equiv 1\pmod 2.\end{array}\right.
 		$$ 
 		\item The relative discriminant of  $L/k$ is given by  $\Delta_{L/k}=\delta_{L/k}\mathcal{O}_k,$ where:
 		$$
 		\delta_{L/k}=\left\{\begin{array}{lll}
 			4a \sqrt{d}, & \text { if } & d \equiv 0\pmod 2,  \\
 			4a \sqrt{d}, & \text { if } & d \equiv 1\pmod 4, b \equiv 0\pmod 2, a+b \equiv 3\pmod 4, \\ 
 			a\sqrt{d},& \text { if } & d \equiv 1\pmod 4, b \equiv 0\pmod 2, a+b \equiv 1\pmod 4, \\ 
 			8a\sqrt{d}, & \text { if } & d \equiv 1\pmod 4, b \equiv 1\pmod 2.\end{array}\right.
 		$$ 
 	\end{enumerate}
 \end{lemma}

 \begin{lemma}[\cite{ref9}]
 	Keep the above notations. If the class number of $k=\mathbb{Q}(\sqrt{d})$ is odd, then $L=\mathbb{Q}(\sqrt{a' \varepsilon_d\sqrt{d}}),$ where 
 	$$
 	a'=
 	\begin{cases} 2a, &\text{ if } d\equiv 1\pmod 4 \text{ and } b\equiv 1\pmod 2,\\
 		a, &\text{otherwise}.\end{cases}
 	$$
 \end{lemma}

 \begin{proposition}[\cite{ref11}]\label{prop 2.4}
 	Let $L$ be an abelian extension of $\mathbb{Q}$ of degree $n$. 
 	If $n=r^s$, where $r$ is a prime number and $s$ is a positive integer, then 
 	$$L^{(*)} =\left(\prod_{p/\delta_L, p \neq r }M_p\right)L,$$
 	where $M_p$ is the unique subfield of degree $e_p$ (the ramification index of $p$ in $L$) over $\mathbb{Q}$ of $\mathbb{Q}(\xi_p)$ the $p$-th cyclotomic field and $\delta_L$ is the discriminant of $L$. 
 \end{proposition}

 \begin{proposition}[\cite{ref15}, p. 160]\label{prop 2.5}
 	If $p$ is a prime number such that $p\equiv 1\pmod 4$, then $\mathbb{Q}(\sqrt{\varepsilon_p^*\sqrt{p}})$ is the quartic subfield of $\mathbb{Q}(\xi_p)$, where
 	$\varepsilon_p^*=\left(\frac{2}{p}\right)\varepsilon_p$ and $\mathbb{Q}(\sqrt{2+\sqrt{2}})$ is the real quartic field of $\mathbb{Q}(\xi_{16})$.	 
 \end{proposition}
 

 We close this section with the following two results.
 \begin{lemma}[\text{\cite{Gr}}]\label{ambiguous class number formula} Let $k/k'$ be a quadratic extension of number fields. If the class number of $k'$ is odd, then the rank of the $2$-class group of $k$ is given by
 	$$r_2({\mathbf{C}l(k)})=t-1-e,$$
 	where  $t$ is the number of  ramified primes (finite or infinite) in the extension  $k/k'$ and $e$ is  defined by   $2^{e}=[E_{k'}:E_{k'} \cap N_{k/k'}(k^*)]$.
 \end{lemma}
 \begin{proposition}[\cite{ref8}]\label{prop 2.7}
 	Let $k/k'$ be a quadratic extension of number fields  and $\mu$ a number of $k'$, coprime with $2$, such that $k=k'(\sqrt{\mu})$. The extension $k/k'$ is unramified  at all finite primes of $k'$ if  and only  if     the two following items  hold:
 	\begin{enumerate}[\indent\rm 1.]
 		\item the ideal generated by $\mu$ is the square of the fractional ideal of $k'$, and 
 		\item there exists a nonzero number $\xi$ of $k'$ verifying $\mu \equiv \xi^2\pmod{4}$.
 	\end{enumerate}
 \end{proposition}
 \section{\bf Genus fields of the imaginary cyclic quartic  fields : $\ K=\mathbb{Q}(\sqrt{-a{\varepsilon}_p \sqrt{p}})$}\label{sec:3}
 
 Let $p$ denote a prime number such that  $p=2$ or  $p \equiv  1\pmod 4$    and $a$ be a positive square-free integer coprime with  $p$. In the present section, we shall investigate the genus field of the imaginary  cyclic quartic fields $K=\mathbb{Q}(\sqrt{-a{\varepsilon}_p \sqrt{p}})$. Note that $K$ is  a $\mathrm{CM}$-field with maximal real subfield $k_0=\mathbb{Q}(\sqrt{p})$. For the construction of the genus field of $K$, we shall need  the factorization of the integer $a$, therefore we put: 
 \begin{eqnarray}
 	a=\displaystyle\prod_{i=1}^n q_i \text{  or } 2\displaystyle\prod_{i=1}^n q_i,
 \end{eqnarray}
 where  $q_1, q_2, \dots, q_n$ are distinct odd primes. Assume    that the Legendre symbols $\left(\dfrac{p}{q_j}\right)=1$ for $1\leq j \leq m$ $( m\leq n$, the case $m=0$ is included here$)$ and $\left(\dfrac{p}{q_j}\right)=-1$ for $m+1 \leq j \leq n$.\\
 Let $e_\ell$ denote   the ramification index of a prime number $\ell$ in $K/\mathbb{Q}$ and put
 $q_j^{*}=(-1)^{\frac{q_j-1}{2}}q_j$.

 \begin{proposition} \label{prop 3.1}
 	Assume that  $p\equiv 5 \pmod 8$, then
 	\begin{enumerate}[\rm 1.]
 		\item $ \text{ If } a=(\displaystyle\prod_{i=1}^n q_i) \equiv 1\pmod4$, we have $$K^{(*)}=K(\sqrt{{q_1}^*}, \sqrt{{q_2}^*}, \dots,\sqrt{{q_n}^*});$$
 		\item $\text{ If }  a=(\displaystyle\prod_{i=1}^n q_i) \equiv 3\pmod 4,$ we have $$ K^{(*)}=K(\sqrt{-1}, \sqrt{{q_1}^*}, \sqrt{{q_2}^*}, \dots, \sqrt{{q_n}^*});$$
 		\item  $\text{ If } a=2\displaystyle\prod_{i=1}^n q_i \; \text{ and } \; (\displaystyle\prod_{i=1}^n q_i) \equiv 1\pmod4$,  we have
 		$$K^{(*)}=K(\sqrt{2}, \sqrt{{q_1}^*}, \sqrt{{q_2}^*}, \dots, \sqrt{{q_n}^*});$$
 		\item $\text{ If } a=2\displaystyle\prod_{i=1}^n q_i \; \text{ and } \; (\displaystyle\prod_{i=1}^n q_i) \equiv 3\pmod4$,  we have
 		$$ K^{(*)}=K(\sqrt{-2}, \sqrt{{q_1}^*}, \sqrt{{q_2}^*}, \dots, \sqrt{{q_n}^*}).$$
 	\end{enumerate}
 \end{proposition}
 \begin{proof}
 	We have $p\equiv 5 \pmod 8$, then by the expression of the absolute discriminant of $K$ (cf. Lemma \ref{lemma 2.1; discriminants}), the prime numbers of $\mathbb{Q}$ that   ramify in $K$ are :
 	\begin{enumerate}[\indent\rm $\bullet$]
 		\item 	the odd prime divisors of $a$ (with  ramification index   equals $2$),
 		\item  the prime $p$ (with $e_p=4$),
 		\item and  $2$ if $a\not\equiv 1\pmod 4$ (with $e_2=2$).
 	\end{enumerate}
 	On the other hand, we know  that  $\mathbb{Q}(\xi_p)/\mathbb{Q}$ is a Galois extension such that $$\mathrm{Gal}(\mathbb{Q}(\xi_p)/\mathbb{Q}) \simeq {(\mathbb{Z}/p\mathbb{Z})}^* \simeq \mathbb{Z}/(p-1)\mathbb{Z},$$ and if $m$ divides $p-1$, then $\mathbb{Q}(\xi_p)$ contains a unique subfield of degree $m$ over $\mathbb{Q}$. Since $p$ is the unique prime which is ramified in $\mathbb{Q}(\xi_p)$, then  $\mathbb{Q}(\sqrt{{-{\varepsilon}}_p\sqrt{p}})$ is the unique subfield of degree 4 of $\mathbb{Q}(\xi_p)$ (Proposition \ref{prop 2.5}). Therefore, by     Proposition \ref{prop 2.4}, we   get: 
 	\begin{enumerate}[\rm $1.$]
 		\item
 		If $a=\displaystyle\prod_{i=1}^n q_i\equiv 1\pmod 4$,  then:
 		
 		$$ K^{(*)}=\mathbb{Q}(\sqrt{{q_1}^*})\mathbb{Q}(\sqrt{{q_2}^*})\dots\mathbb{Q}(\sqrt{{q_n}^*})\mathbb{Q}(\sqrt{{-{\varepsilon}}_p\sqrt{p}})K=K(\sqrt{{q_1}^*}, \sqrt{{q_2}^*}, \dots,\sqrt{{q_n}^*}).$$
 		\item If  $a=\displaystyle\prod_{i=1}^n q_i\equiv 3\pmod 4$, then:
 		$$ K^{(*)}=\mathbb{Q}(\sqrt{{q_1}^*})\mathbb{Q}(\sqrt{{q_2}^*})\dots\mathbb{Q}(\sqrt{{q_n}^*})\mathbb{Q}(\sqrt{{-{\varepsilon}}_p\sqrt{p}})K=K(\sqrt{-1}, \sqrt{{q_1}^*}, \sqrt{{q_2}^*}, \dots, \sqrt{{q_n}^*}).$$  
 		\item If $a=2\displaystyle\prod_{i=1}^n q_i$ and $\displaystyle\prod_{i=1}^n q_i\equiv 1\pmod 4$,  then:
 		$$ K^{(*)}=\mathbb{Q}(\sqrt{{q_1}^*})\mathbb{Q}(\sqrt{{q_2}^*})\dots\mathbb{Q}(\sqrt{{q_n}^*})\mathbb{Q}(\sqrt{{-{\varepsilon}}_p\sqrt{p}})K=K(\sqrt{2}, \sqrt{{q_1}^*}, \sqrt{{q_2}^*}, \dots, \sqrt{{q_n}^*}).$$
 		\item  If $a=2\displaystyle\prod_{i=1}^n q_i$ and 
 		$\displaystyle\prod_{i=1}^n q_i\equiv 3\pmod 4$,  then:
 		$$ K^{(*)}=\mathbb{Q}(\sqrt{{q_1}^*})\mathbb{Q}(\sqrt{{q_2}^*})\dots\mathbb{Q}(\sqrt{{q_n}^*})\mathbb{Q}(\sqrt{{-{\varepsilon}}_p\sqrt{p}})K=K(\sqrt{-2}, \sqrt{{q_1}^*}, \sqrt{{q_2}^*}, \dots, \sqrt{{q_n}^*}).$$
 	\end{enumerate}
 	Which completes the proof.
 \end{proof}  
 
 We similarly prove the following two propositions.
 \begin{proposition} 
 	Assume that  $p\equiv 1 \pmod 8$, then
 	\begin{enumerate}[\indent\rm 1.] 
 		\item If  $a=\displaystyle\prod_{i=1}^n q_i\equiv 3 \pmod 4$, then   $K^{(*)}=K(\sqrt{{q_1}^*}, \sqrt{{q_2}^*}, \dots, \sqrt{{q_n}^*}).$
 		\item If $a=\displaystyle\prod_{i=1}^n q_i\equiv 1 \pmod 4$, then $K^{(*)}=K(\sqrt{-1}, \sqrt{{q_1}^*}, \sqrt{{q_2}^*}, \dots, \sqrt{{q_n}^*}).$
 		\item If $a=2\displaystyle\prod_{i=1}^n q_i$    and   $\displaystyle\prod_{i=1}^n q_i\equiv 1 \pmod 4$,  then 
 		$K^{(*)}=K(\sqrt{-2}, \sqrt{{q_1}^*}, \sqrt{{q_2}^*}, \dots, \sqrt{{q_n}^*}).$
 		\item If $a=2\displaystyle\prod_{i=1}^n q_i$   and   $\displaystyle\prod_{i=1}^n q_i\equiv 3 \pmod 4$,  then $K^{(*)}=K(\sqrt{2}, \sqrt{{q_1}^*},\sqrt{{q_2}^*}, \dots,\sqrt{{q_n}^*}).$
 	\end{enumerate} 
 \end{proposition}
 
 \begin{proposition} 
 	If $p=2, \text{ then } K^{(*)}=K(\sqrt{{q_1}^*}, \sqrt{{q_2}^*}, \dots, \sqrt{{q_n}^*}).$
 	
 \end{proposition}

 \section{\bf Hilbert genus fields of the imaginary  cyclic quartic fields : $\ K=\mathbb{Q}(\sqrt{-a{\varepsilon}_p \sqrt{p}})$}\label{sec:4}
 Keep the notations of the previous sections.
 At this stage, we can start the construction of the Hilbert genus field of the imaginary cyclic quartic field  $  K=\mathbb{Q}(\sqrt{-a{\varepsilon}_p \sqrt{p}})$, 
 however, we must first expose some more ingredients of our proofs. We have the following lemmas.
 \begin{lemma}\label{lemma 4.1}
 	Let $k/k'$ be a quadratic extension of number fields such that the class number of $k'$ is odd.   Let $\Delta$ denote the multiplicative group  such that ${k^*}^2 \subset {\Delta} \subset k^*$ and $k(\sqrt{\Delta})$ is the Hilbert genus field of $k$  $($\textnormal{cf}.  \S \ref{sec:1}$)$. Then 
 	$$r_2(\Delta/{k^*}^2)=t-e-1,$$
 	where $t$ and $e$ are defined in Lemma \ref{ambiguous class number formula}.
 \end{lemma}
 \begin{proof}Put  $ \mathrm{G}=\mathrm{Gal}(H(k)/k)$.
 	By the definition of $E(k) $   $($\textnormal{cf}.  \S \ref{sec:1}$)$ we have: $$\mathbf{C}l(k)/\mathbf{C}l(k)^2   \simeq \mathrm{G}/\mathrm{G}^2 \simeq \mathrm{Gal}(E(k)/k),$$
 	Since $E(k)=k(\sqrt{\Delta})$, then by class field theory:
 	$$r_2(\Delta/{k^*}^2)= \log _2[k(\sqrt{\Delta}): k]=\log _2[E(k):k]=r_2(\mathbf{C}l(k)).$$
 	So the result by the ambiguous class number formula (Lemma \ref{ambiguous class number formula}).
 \end{proof}
 
 \begin{lemma}[\cite{ref13}]
 	Let     $p\equiv 1\pmod 4$  be a prime number and $\varepsilon_p$ is the fundamental unit of $\mathbb{Q}(\sqrt{p})$. Then, there are two natural integers u and v such that:$$\varepsilon_p^\lambda = u+v\sqrt{p},\hspace{1cm}  u\equiv 0\pmod 2, \hspace{1cm} v\equiv 1\pmod 4,$$
 	where $\lambda=\begin{cases}1,& \text{ if } p\equiv 1\pmod 8;\\
 		3, & \text{ if }  p\equiv 5\pmod 8.\end{cases}$
 \end{lemma}

 \begin{lemma}\label{lemma 4.3}
 	Let $p$ and $q$ be two different prime numbers such that $p \equiv 1\pmod 4$, $q \equiv \pm 1 \pmod 4$ and $\left(\dfrac{p}{q}\right)=1$. Then, there exist two natural integers $x$ and $y$ such that: $$ x^2-py^2=q^{\lambda h},$$
 	where $h$ is the class number of  $\mathbb{Q}(\sqrt{p})$ and  $\lambda=\begin{cases}1,& \text{ if } p\equiv 1\pmod 8;\\
 		3, & \text{ if }  p\equiv 5\pmod 8.\end{cases}$\\
 	Furthermore,  we have:$$\begin{cases} x\equiv 1\pmod 2, \hspace{0.5cm} y\equiv 0\pmod 2, \text{ if }  q\equiv 1\pmod 4;\\
 		x\equiv 0\pmod 2, \hspace{0.5cm} y\equiv 1\pmod 2, \text{ if } q\equiv 3\pmod 4.\end{cases}$$
 \end{lemma}
 \begin{proof} 
 	Since $(\frac{p}{q})=1$, then   $q$ splits in $k_0=\mathbb{Q}(\sqrt{p})$. Thus, there exist two prime ideals $\mathcal{H}_1$ and $\mathcal{H}_2$ of $\mathcal{O}_{k_0}$, such that $q\mathcal{O}_{k_0}=\mathcal{H}_1\mathcal{H}_2$ and $\sigma (\mathcal{H}_1)=\mathcal{H}_2$ ($\sigma$ is the generator of Galois group  of $k_0$). Thus, 
 	$q^h\mathcal{O}_{k_0}=\mathcal{H}_1^{h}\mathcal{H}_2^{h}$, where $h$ is the class number of $k_0$. Since $\mathcal{H}_1^{h}$ and $\mathcal{H}_2^{h}$ are two  principal ideals,  we can choose two natural integers $x$ and $y$ such that $\mathcal{H}_1^{\lambda h}=(x+\sqrt{p}y)$,  $\mathcal{H}_2^{\lambda h}=(x-\sqrt{p}y)$ and $x^2-py^2> 0$. Then $q^{\lambda h}\mathcal{O}_{k_0}=(x^2-py^2)$.
 	Therefore, $x^2-py^2=\eta q^{\lambda h}$, for a certain unit $\eta \in E_{k_0}$. Thus  $\eta \in E_{k_0} \cap \mathbb{Q}$. Since $E_{k_0}=\langle-1,\;\varepsilon_p \rangle$ and $x^2-py^2> 0$, it follows that $\eta =1$ and so $x^2-py^2= q^{\lambda h}$, which gives the first part of the lemma. 
 	If we assume that, $q\equiv 1\pmod 4$, we get $x^2-y^2 \equiv 1\pmod 4$.  
 	By checking all the possibilities of $x^2-y^2\pmod{4}$  we deduce that, $x \equiv 1\pmod 2$ and $y \equiv 0\pmod 2$. If  $q \equiv 3\pmod 4$, we get $x^2-y^2 \equiv 3\pmod 4$.
 	Hence, $x \equiv 0\pmod 2$ and $y \equiv 1\pmod 2$. Which completes the proof.
 \end{proof}
 
 \begin{lemma}\label{lemma 4.4}
 	Let $k=k'(\sqrt{\mu})$ be a quadratic extension of number fields  and $\beta \in k'$. Then,
 	$\beta \in {k}^2 \text{ if  and only  if }\beta \in {k'}^2\text{ or }\mu\beta \in {k'}^2.$
 \end{lemma}
 \begin{proof}
 	If $\beta \in k^2$, then $\beta = {(a+b\sqrt{\mu})}^2$, for some $a$, $b \in {k'}$, so $ \beta =a^2+\mu b^2+2ab\sqrt{\mu}$. Since $\beta \in {k'}$, we have $a=0$ or $b=0$. If $b=0$, then $\beta =a^2 \in {k'}^2$. If $a=0$, then $\beta= \mu b^2$.  Therefore $\mu\beta={(\mu b)}^2\in {k'}^2$. The reciprocal implication is evident.
 \end{proof}
 
 Now, after all the above preparations,  we can state our first main theorem. We keep the notations in the beginning of the previous section and  we shall put $E=E(K)$.

 \begin{theorem}\label{first main theorem}
 	Let $p\equiv 5\pmod 8$ be a prime number and denote by $E$ the Hilbert genus field of  $K=\mathbb{Q}(\sqrt{-a{\varepsilon}_p \sqrt{p}})$. We have:
 	\begin{enumerate}[\indent\rm 1.]
 		\item If $a\equiv 1\pmod 4$, then:
 		$$E=K(\sqrt{{q_1}^*}, \sqrt{{q_2}^*}, \dots, \sqrt{{q_n}^*}, \sqrt{{\alpha_1}^*}, \sqrt{{\alpha_2}^*}, \dots, \sqrt{{\alpha_m}^*});$$
 		\item If $a\equiv 3\pmod 4$, then:
 		$$E=K(\sqrt{-1}, \sqrt{{q_1}^*}, \sqrt{{q_2}^*}, \dots, \sqrt{{q_n}^*}, \sqrt{{\alpha_1}^*},\sqrt{{\alpha_2}^*},\dots, \sqrt{{\alpha_m}^*});$$
 		\item If $a=2\displaystyle\prod_{i=1}^n q_i$ and $\displaystyle\prod_{i=1}^n q_i\equiv 1\pmod 4$, then:
 		
 		$$ E=K(\sqrt{2}, \sqrt{{q_1}^*}, \sqrt{{q_2}^*}, \dots,\sqrt{{q_n}^*},\sqrt{{\alpha_1}^*},\sqrt{{\alpha_2}^*}, \dots, \sqrt{{\alpha_m}^*});$$
 		
 		\item  If  $a=2\displaystyle\prod_{i=1}^n q_i$ and $\displaystyle\prod_{i=1}^n q_i\equiv 3\pmod 4$, then: $$ E=K(\sqrt{-2}, \sqrt{{q_1}^*}, \sqrt{{q_2}^*}, \dots, \sqrt{{q_n}^*}, \sqrt{{\alpha_1}^*},\sqrt{{\alpha_2}^*},\dots, \sqrt{{\alpha_m}^*});$$
 	\end{enumerate}
 	where:
 	$$ \left\{\begin{array}{lll}
 		{q_i}^*=(-1)^{\frac{q_i-1}{2}}q_i,\quad (1 \leqslant i \leqslant n),\\
 		\alpha_j=x_j+y_j\sqrt{p},\text{  }(1 \leqslant j \leqslant m),\; x_j \text{ and } y_j \text{ are the integers  given by Lemma \ref{lemma 4.3},  }\\
 		\quad\qquad \qquad \qquad  \qquad \qquad \qquad \text{such that } x_j^2-py_j^2=q_j^{\lambda h} ,\\
 		\alpha_j^*=\begin{cases}(-1)^{\frac{x_j+y_j-1}{2}}\alpha_j, \hspace{0.5cm}\text{ if }  q_j\equiv 1\pmod 4,\\
 			(-1)^{\frac{x_j+y_j-1}{2}}\alpha_j \sqrt{p},\hspace{0.5cm} \text{ if } q_j\equiv 3\pmod 4,\end{cases} \quad (1\leqslant j \leqslant m). 
 	\end{array}\right. $$
 \end{theorem}
 
 \begin{proof}Note that $K/k_0$ is a quadratic extension with the class number of $k_0=\mathbb{Q}(\sqrt{p})$ is odd. So we are in the conditions of Lemma \ref{lemma 4.1}.
 	Let $\Delta$ be the multiplicative group  such that   $E=K(\sqrt{\Delta})$. Thus  $r_2(\Delta/{K^*}^2)=t-e-1$. Let    $p_{\infty}$ and $\bar p_{\infty}$ denote the  infinite primes of $k_0$ which are respectively corresponding to the $\mathbb{Q}$-embeddings:
 	
 	$$
 	\begin{array}{rcl}
 		i_{p_{\infty}}:  k_0  &\hookrightarrow &\mathbb{R} \\
 		\sqrt{p} &\longmapsto &-\sqrt{p}
 	\end{array} \hspace{1cm} \text{ and } \hspace{1cm}
 	\begin{array}{rcl}
 		i_{\bar{p}_{\infty}}:  k_0  &\hookrightarrow & \mathbb{R} \\
 		\sqrt{p} &\longmapsto & \sqrt{p}
 	\end{array}
 	$$ 
 	Set $i_{{p}_{\infty}}(\varepsilon_{p})=\bar{\varepsilon}_p$. Note that  $N_{k_0/\mathbb{Q}}(\varepsilon_p)=\varepsilon_p \bar{\varepsilon}_p=-1$.
 	By the definition of Hilbert symbol (cf. \cite{ref5}),   we have: 
 	$$
 	\begin{array}{rcl}
 		\left(\frac{-1,-a\varepsilon_{p}\sqrt{p} }{p_{\infty}}\right)
 		&=&i_{p_{\infty}}^{-1}(( i_{p_{\infty}}( -1), i_{p_{\infty}}(-a\varepsilon_{p}\sqrt{p}) ) _{p_{\infty}})\\
 		&=& i_{p_{\infty}}^{-1}(( -1, a\bar{\varepsilon}_{p}\sqrt{p} ) _{p_{\infty}})\\
 		&=& i_{p_{\infty}}^{-1}(-1)\\
 		&=& -1,
 	\end{array}
 	$$
 	$$
 	\begin{array}{rcl}
 		\left(\frac{\varepsilon_{p},-a\varepsilon_{p}\sqrt{p} }{{p}_{\infty}}\right)
 		&=&i_{{p}_{\infty}}^{-1}(( i_{{p}_{\infty}}(\varepsilon_{p}), i_{{p}_{\infty}}(-a\varepsilon_{p}\sqrt{p}) ) _{{p}_{\infty}})\\
 		&=& i_{{p}_{\infty}}^{-1}(( \bar{\varepsilon}_{p}, a \bar{\varepsilon}_{p}\sqrt{p} ) _{{p}_{\infty}})\\
 		&=& i_{{p}_{\infty}}^{-1}(-1)\\
 		&=& -1,
 	\end{array}
 	$$  
 	$$
 	\begin{array}{rcl}
 		\left(\frac{-\varepsilon_{p},-a\varepsilon_{p}\sqrt{p} }{\bar{p}_{\infty}}\right)
 		&=&i_{\bar{p}_{\infty}}^{-1}(( i_{\bar{p}_{\infty}}(-\varepsilon_{p}), i_{\bar{p}_{\infty}}(-a\varepsilon_{p}\sqrt{p}) ) _{\bar{p}_{\infty}})\\
 		&=& i_{\bar{p}_{\infty}}^{-1}((-\varepsilon_{p}, -a\varepsilon_{p}\sqrt{p} ) _{\bar{p}_{\infty}})\\
 		&=& i_{\bar{p}_{\infty}}^{-1}(-1)\\
 		&=& -1.
 	\end{array}
 	$$
 	It follows, by the Hasse norm theorem, that  $-1$, $\varepsilon_p$ and $-\varepsilon_p$ are not in $N_{K/k_0}(K^*)$. Thus, $e=2$. Therefore,  to compute $r_2(\Delta/{K^*}^2)$, it suffices to determine the number of   prime ideals of $k_0$ which ramify in $K$.
 	
 	\begin{enumerate}[\rm {Case }1.]
 		\item Assume that  $a\equiv 1\pmod 4$:\\
 		In this case, we have $\delta_{K/k_0}= a\sqrt{p}$ (Lemma \ref{lemma 2.1; discriminants}). Thus the finite primes   of $k_0$ which ramify   in $K$ are the prime divisors of $a$ in $k_0$ and the prime ideal $(\sqrt{p})$. 
 		Since $\left(\dfrac{p}{q_j}\right) =1$, for $1\leq j \leq m$, then $q_j$   splits in $k_0$. Thus, we have:
 		$$r_2(\Delta/{K^*}^2) = (n+m+1)+2-2-1=n+m.$$ 
 		On the other hand, to  explicitly determine $E$, it suffices  to determine the set of generators for the finite group $(\Delta/{K^*}^2)$. For this, we consider the set:
 		$$\mathbb{B}=\{q_1^*,q_2^*,\dots,q_n^*,\alpha_1^*,\alpha_2^*,\dots,\alpha_m^*\}.$$ 
 		Let us  show that the elements of $\mathbb{B}$ are linearly independent   modulo ${K^*}^2$ (provided that the notion
 		of linear independence is translated to a multiplicative setting: $\alpha_1,..., \alpha_s$ are multiplicatively independent if $\alpha_1^{m_1}... \alpha_s^{m_s}= 1$ implies that
 		$m_i=0$, for all $i$).

 		We consider the element $\beta =\left(\displaystyle\prod_{i=1}^n {q_i^*}^{a_i}\right)\left(\displaystyle\prod_{j=1}^m {\alpha_j^*}^{b_j}\right)$, where $a_i, b_j \in \{0,1\}$ and are not all zero. Suppose that $\beta \in {K^*}^2 .$\\
 		$\bullet$ If we assume that $\forall \, j \in \{1,2,\dots,m\}\;\,b_j=0$, we get  $\beta =\left(\displaystyle\prod_{i=1}^n {q_i^*}^{a_i}\right)$ $\in {K^*}^2$ which implies $\sqrt{\beta} \in K^*$. Thus, $\mathbb{Q}(\sqrt{\beta})$ is a quadratic subfield of $ K $, which is impossible, because $ K $ has only one quadratic subfield that is $k_0=\mathbb{Q}(\sqrt{p})$.\\
 		$\bullet$ Assume that  $\exists \, j \in \{1,\dots,m\}$ with $b_j\neq 0$. Note that $\beta \in k_{0}^*$ and $\beta\in {K^*}^2$. So by Lemma \ref{lemma 4.4}, we have $\beta \in {k_{0}^*}^2 \text{ or }  \delta \beta \in {k_{0}^*}^2 \text{ where } \delta = -a\varepsilon_p \sqrt{p}$. We shall discuss each case.\\
 		Suppose that $\beta \in {k_{0}^*}^2.$ Then, we have $N_{k_0/\mathbb{Q}}(\beta)=\beta \beta^{\sigma} \in {\mathbb{Q}^*}^2$ $($with $\mathrm{Gal}(k_0/\mathbb{Q})$=$\{1,\sigma \}$$)$. Put $N=N_{k_0/\mathbb{Q}}$. We have $\forall \, i \in \{1,2,\dots,n\}$,  $N({q_i}^*)=q_i^2$ and  $\forall \, j \in \{1,2,\dots,m\}$: $$N({\alpha_j}) = x_j^2-py_j^2 =  q_j^{3h}  = q_jz_j^2, $$
 		where $z_j=(q_j)^{\frac{3h-1}{2}}$ (where $h$  is the class number of $k_0$, note that here $h$   is an odd number) and  $N(\sqrt{p})=-p$, so:
 		$$
 		N({\alpha_j}^*)=\begin{cases} q_jz_j^2, \hspace{1cm} \text{ if } q_j\equiv 1\pmod 4,\\
 			-pq_jz_j^2, \hspace{0.5cm}  \text{ if } q_j\equiv 3\pmod 4.\end{cases}
 		$$
 		Thus 
 		$$
 		N(\beta)=(-p)^l\left(\displaystyle\prod_{i=1}^n {q_i}^{a_i}\right)^2\left(\displaystyle\prod_{j=1}^m {q_j}^{b_j}\right)\left(\displaystyle\prod_{j=1}^m {z_j}^{b_j}\right)^2,
 		$$
 		where $l$ is the number of $q_j$ such that $q_j\equiv 3\pmod 4$.\\ 
 		If $l$ is even, $N(\beta)=\left(\displaystyle\prod_{j=1}^m {q_j}^{b_j}\right)Z^2 \in {{\mathbb{Q}}^*}^2$ where $Z\in \mathbb{Q}$, then $\left(\displaystyle\prod_{j=1}^m {q_j}^{b_j}\right)\in {{\mathbb{Q}}^*}^2$.\\
 		If $l$ is odd, $N(\beta)=-p\left(\displaystyle\prod_{j=1}^m {q_j}^{b_j}\right)Z^2 \in {{\mathbb{Q}}^*}^2$, so $-p\left(\displaystyle\prod_{j=1}^m {q_j}^{b_j}\right)\in {{\mathbb{Q}}^*}^2$. \\
 		Both cases are not possible because $ p, q_1, q_2, \dots,q_m$ are distinct prime numbers.\\
 		If now $\delta \beta \in {k_{0}^*}^2$,  then $N(\delta \beta)$ =$ N(\beta)N(\delta)\in  {{\mathbb{Q}}^*}^2 $. Since $N(\varepsilon_p)=-1$, we have  $N(\delta )=N(-a)N(\varepsilon_p)N(\sqrt{p})=pa^2$. Thus :
 		$$
 		N(\delta \beta)=pa^2(-p)^l\left(\displaystyle\prod_{i=1}^n {q_i}^{a_i}\right)^2\left(\displaystyle\prod_{j=1}^m {q_j}^{b_j}\right)\left(\displaystyle\prod_{j=1}^m {z_j}^{b_j}\right)^2,
 		$$
 		if $l$ is even, $N(\delta\beta)=p\left(\displaystyle\prod_{j=1}^m {q_j}^{b_j}\right)Y^2 \in {{\mathbb{Q}}^*}^2$, where $Y\in \mathbb{Q}$, then $p\displaystyle\prod_{j=1}^m {q_j}^{b_j}\in {{\mathbb{Q}}^*}^2,$
 		if $l$ is odd, $N(\delta\beta)=-\left(\displaystyle\prod_{j=1}^m {q_j}^{b_j}\right)Y^2 \in {{\mathbb{Q}}^*}^2$, then $-\displaystyle\prod_{j=1}^m {q_j}^{b_j}\in {{\mathbb{Q}}^*}^2,$
 		Both of these cases are impossible. Hence the elements of $\mathbb{B}$ are linearly  independent modulo ${K^*}^2$.

 		On the other hand, according to Proposition \ref{prop 3.1}, the genus field  of $K$ is $K^{(*)}=\left(\sqrt{{q_1}^*},\sqrt{{q_2}^*},\dots,\sqrt{{q_n}^*}\right)$. So for $1\leq i\leq n$,   $K(\sqrt{q_i^*})/K$ is an unramified extension, and by Lemma \ref{lemma 4.3}, for $1\leq j \leq m$,  $q_jO_k=\mathcal{H}_1^j\mathcal{H}_2^j$, where $\mathcal{H}_1^j$ and $\mathcal{H}_2^j$ are the prime ideals of $\mathcal{O}_{k_0}$ above $q_j$, since $\mathcal{H}_1^j$ is ramified in $K$ and $\mathcal{H}_1^{j^{3h}}=(x_j+y_j\sqrt{p})=(\alpha _j)$, then $(\alpha _j)$ is the square of the fractional ideal of $K$.
 		Note also that if $q_j\equiv 1\pmod 4$  for some $  1\leq j \leq m $, we have $x_j$ is odd and $y_j$ is even. So $\alpha_j = x_j+y_j\sqrt{p}\equiv x_j+y_j+\frac{-1+\sqrt{p}}{2}2y_j\pmod 4$  $\equiv x_j+y_j\pmod 4$  $\equiv \pm 1 \pmod 4$. Therefore, $\alpha_j^* \equiv  1 \pmod 4$ and the ideal generated by $\alpha_j^*$ is the square of the fractional ideal of $K$. Thus, by Proposition \ref{prop 2.7}, $K(\sqrt{\alpha _j^*})/K$ is also an unramified extension (the proof of the case $q_j\equiv 3\pmod 4$ is analogous). Therefore:
 		$$\mathbb{B}=\{q_1^*,q_2^*,...,q_n^*,\alpha_1^*,\alpha_2^*,\dots,\alpha_m^*\},$$
 		is a representative set of $\Delta /{K^*}^2$. Hence :
 		$$E=K(\sqrt{{q_1}^*},\sqrt{{q_2}^*},\dots,\sqrt{{q_n}^*},\sqrt{{\alpha_1}^*},\sqrt{{\alpha_2}^*},\dots,\sqrt{{\alpha_m}^*}),$$
 		is the Hilbert genus field of $K$.\\
 		
 		\item  Case: $a\equiv 3\pmod 4$:\\
 		In this case, we have $\delta_{K/k_0}= 4a\sqrt{p}$. So the   prime  ideals of $k_0$ which ramify in $K$ are the prime divisors of $a$ in $k_0$, the ideal $(\sqrt{p})$ and also the prime  ideal $(2)$. Thus, $r_2(\Delta/{K^*}^2)$ = $n+m+1$.\\
 		We consider the set:
 		$$\mathbb{B}=\{-1,q_1^*,q_2^*,\dots,q_n^*,\alpha_1^*,\alpha_2^*,\dots,\alpha_m^*\},$$
 		We proceed   as in the first case  and we show that the elements of   $\mathbb{B}$ are linearly independent modulo ${K^*}^2$. Noting that the genus fields of $K$ is:  $${K^{(*)}}=K(\sqrt{-1}, \sqrt{{q_1}^*},\sqrt{{q_2}^*},\dots,\sqrt{{q_n}^*}),$$ we deduce that the extensions $K(\sqrt{-1})/K$ and $K(\sqrt{q_i^*})/K$ (for $1\leq i\leq n$) are unramified, and in a similar way as in the    previous case,  we prove that the extensions  $K(\sqrt{\alpha _j^*})/K$ are also unramified. Thus, $\mathbb{B}$ is a representative set of $\Delta /{K^*}^2$. Therefore,
 		
 		$$E=K(\sqrt{-1},\sqrt{{q_1}^*},\sqrt{{q_2}^*},\dots,\sqrt{{q_n}^*},\sqrt{{\alpha_1}^*},\sqrt{{\alpha_2}^*},\dots,\sqrt{{\alpha_m}^*}).$$
 		\item  $a=2\displaystyle\prod_{i=1}^n q_i$ and $\displaystyle\prod_{i=1}^n q_i\equiv 1\pmod 4$.\\
 		In this case, we have $\delta_{K/k_0}= 4a\sqrt{p}$. So the   prime ideals of $k_0$ which ramify in $K$ are the prime divisors of $a$ in $k_0$, the ideal $(\sqrt{p})$ and the
 		prime ideal $(2)$. Then $r_2(\Delta/{K^*}^2)$ = $n+m+1$, and since the genus fields of $K$ is: $${K^{(*)}}=K(\sqrt{2}, \sqrt{{q_1}^*},\sqrt{{q_2}^*}\text{, . . .,}\sqrt{{q_n}^*}),$$ so we shall consider the set:
 		$$\mathbb{B}=\{2,q_1^*,q_2^*,\dots,q_n^*,\alpha_1^*,\alpha_2^*,...,\alpha_m^*\}.$$
 		Similarly to the previous cases, we show that $\mathbb{B}$ is  a representative set of $\Delta /{K^*}^2$, then:
 		$$E=K(\sqrt{2},\sqrt{{q_1}^*},\sqrt{{q_2}^*},\dots,\sqrt{{q_n}^*},\sqrt{{\alpha_1}^*},\sqrt{{\alpha_2}^*},\dots,\sqrt{{\alpha_m}^*}).$$
 	\end{enumerate}	
 	Using similar techniques as in the previous case, we prove the fourth item.
 	Which completes the proof.
 \end{proof}

 \begin{theorem}\label{second main theorem}
 	Let $p\equiv 1\pmod 8$ be a prime number and denote by $E$ the Hilbert genus field of  $K=\mathbb{Q}(\sqrt{-a{\varepsilon}_p \sqrt{p}})$. We have:
 	\begin{enumerate}[\indent\rm 1.]
 		\item If $a\equiv 3\pmod 4$, then:
 		$$E=K(\sqrt{{q_1}^*},\sqrt{{q_2}^*},\dots,\sqrt{{q_n}^*},\sqrt{{\alpha_1}^*},\sqrt{{\alpha_2}^*},\dots,\sqrt{{\alpha_m}^*});$$
 		\item If $a\equiv 1\pmod 4$, then:
 		$$E=K(\sqrt{-1},\sqrt{{q_1}^*},\sqrt{{q_2}^*},\dots,\sqrt{{q_n}^*},\sqrt{{\alpha_1}^*},\sqrt{{\alpha_2}^*},\dots,\sqrt{{\alpha_m}^*},\sqrt{\varepsilon_p});$$
 		\item If  $a=2\displaystyle\prod_{i=1}^n q_i$ and $\displaystyle\prod_{i=1}^n q_i\equiv 1\pmod 4$, then:
 		
 		$$ E=K(\sqrt{-2},\sqrt{{q_1}^*},\sqrt{{q_2}^*},\dots,\sqrt{{q_n}^*},\sqrt{{\alpha_1}^*},\sqrt{{\alpha_2}^*},\dots,\sqrt{{\alpha_m}^*},\sqrt{\varepsilon_p});$$
 		\item If  $a=2\displaystyle\prod_{i=1}^n q_i$ and $\displaystyle\prod_{i=1}^n q_i\equiv 3\pmod 4$, then:
 		$$ E=K(\sqrt{2},\sqrt{{q_1}^*},\sqrt{{q_2}^*},\dots,\sqrt{{q_n}^*},\sqrt{{\alpha_1}^*},\sqrt{{\alpha_2}^*},\dots,\sqrt{{\alpha_m}^*},\sqrt{\varepsilon_p}).$$
 	\end{enumerate}
 \end{theorem}
 \begin{proof}
 	The proof of the case  $a\equiv 3\pmod 4$  is similar to that of the first item of Theorem \ref{first main theorem}, so we let it for the reader. 
 	
 	In the case where $a\equiv 1\pmod 4$, we have $\delta_{K/k_0}= 4a\sqrt{p}$, then the   prime ideals of $k_0$ that ramify in $K$ are the prime divisors of $a$ in $k_0$, 
 	the ideal $(\sqrt{p})$, $2_1$ and $2_2$,    where $2_1$ and $2_2$ are the two prime ideals of $k_0$ above $2$. Thus, $r_2(\Delta/{K^*}^2)$ = $n+m+2$. On the other hand, the genus field of $K$ in this case is  $K^{(*)}=K\left(\sqrt{-1}, \sqrt{{q_1}^*},\sqrt{{q_2}^*},\dots,\sqrt{{q_n}^*}\right)$. As in the proof of   
 	Theorem \ref{first main theorem}, we shall consider the set:  
 	$$\mathbb{B}=\{-1,q_1^*,q_2^*,\dots,q_n^*,\alpha_1^*,\alpha_2^*,\dots,\alpha_m^*,\varepsilon_p \},$$
 	where $\varepsilon_p$ is the fundamental unit of $k_0$ and we shall   show that the elements of $\mathbb{B}$ are linearly independent modulo ${K^*}^2$. Let  $\beta =(-1)^a{\varepsilon_p}^b\left(\displaystyle\prod_{i=1}^n {q_i^*}^{a_i}\right)\left(\displaystyle\prod_{j=1}^m {\alpha_j^*}^{b_j}\right),$ where $a, b, a_i, b_j \in\{0,1\}$ and they are not all zero.
 	Assume that $\beta \in {K^*}^2$.\\
 	$\bullet$ If $b=0$, then $\beta =(-1)^a\left(\displaystyle\prod_{i=1}^n {q_i^*}^{a_i}\right)\left(\displaystyle\prod_{j=1}^m {\alpha_j^*}^{b_j}\right)$, following the same reasoning in the proof of  Theorem \ref{first main theorem}  we get a contradiction.\\

 	\noindent$\bullet$ If $b=1$ we will distinguish two cases:
 	\begin{enumerate}[\indent\rm 1.]
 		\item Assume that $\forall j\in \{1,2,\dots,m\}, b_j=0$, $\beta =(-1)^a\varepsilon_p\left(\displaystyle\prod_{i=1}^n {q_i^*}^{a_i}\right)$, then by Lemma \ref{lemma 4.4}, $\beta \in {k_{0}^*}^2$ or $\delta\beta\in {k_{0}^*}^2 $. If $\beta \in {k_{0}^*}^2$,  then  $N(\beta)=-\left(\displaystyle\prod_{i=1}^n {q_i^*}^{a_i}\right)^2\in {\mathbb{Q}}^2$ (in fact, $N(\varepsilon_p)=-1$). Thus $-1$ is a square in $\mathbb{Q}$, which is impossible. If $\delta\beta\in {k_{0}^*}^2$, $N(\delta\beta)=-pa^2\left(\displaystyle\prod_{i=1}^n {q_i^*}^{a_i}\right)^2\in {\mathbb{Q}}^2$, then $-p$ is a square in $\mathbb{Q}$ and which is also   impossible. 
 		\item Assume that $\exists j\in\{1,2,\dots,m\}, b_j\neq 0$, then $\beta =(-1)^a{\varepsilon_p}\left(\displaystyle\prod_{i=1}^n {q_i^*}^{a_i}\right)\left(\displaystyle\prod_{j=1}^m {\alpha_j^*}^{b_j}\right)$.
 		If  $\beta \in {k_{0}^*}^2$, then by keeping the same previous notations we obtain :
 		$$N(\beta)=-(-p)^l\left(\displaystyle\prod_{i=1}^n {q_i}^{a_i}\right)^2\left(\displaystyle\prod_{j=1}^m {q_j}^{b_j}\right)\left(\displaystyle\prod_{j=1}^m {z_j}^{b_j}\right)^2,$$
 		Thus:
 		$$N(\beta )=\begin{cases} -\left(\displaystyle\prod_{j=1}^m {q_j}^{b_j}\right)Z^2 \in {{\mathbb{Q}}^*}^2,
 			\text{ if } l \text{ is even,} \\
 			-p\left(\displaystyle\prod_{j=1}^m {q_j}^{b_j}\right)Z^2 \in {{\mathbb{Q}}^*}^2, \text{ if } l \text{ is odd,}\end{cases}$$  
 		where $Z\in \mathbb{Q}$, so $-\left(\displaystyle\prod_{j=1}^m {q_j}^{b_j}\right) \in {{\mathbb{Q}}^*}^2$ or $-p\left(\displaystyle\prod_{j=1}^m {q_j}^{b_j}\right) \in {{\mathbb{Q}}^*}^2$, and this is impossible. 
 		
 		\noindent If  now $\delta\beta \in {k_{0}^*}^2$, we get:
 		$$
 		N(\delta\beta )=\begin{cases} p\left(\displaystyle\prod_{j=1}^m {q_j}^{b_j}\right)Y^2 \in {{\mathbb{Q}}^*}^2,
 			\text{ if } l \text{ is even,} \\
 			-\left(\displaystyle\prod_{j=1}^m {q_j}^{b_j}\right)Y^2 \in {{\mathbb{Q}}^*}^2,
 			\text{ if } l \text{ is odd,}\end{cases}
 		$$  
 		where $Y\in \mathbb{Q}$, so $-\left(\displaystyle\prod_{j=1}^m {q_j}^{b_j}\right) \in {{\mathbb{Q}}^*}^2$ or $p\left(\displaystyle\prod_{j=1}^m {q_j}^{b_j}\right) \in {{\mathbb{Q}}^*}^2$, which is also impossible. Therefore, the elements of $\mathbb{B}$ are linearly independent modulo ${K^*}^2$. It is easy seen by \cite[p. 67]{cohn}  that the extension  $K(\sqrt{\varepsilon_p})/K$ is    unramified.
 		Thus, by the above discussion,     $\mathbb{B}$ is  a representative set of $\Delta /{K^*}^2$ and the Hilbert genus field of $K$ is:
 		$$E=K\left(\sqrt{-1},\sqrt{{q_1}^*},\sqrt{{q_2}^*},\dots,\sqrt{{q_n}^*},\sqrt{{\alpha_1}^*},\sqrt{{\alpha_2}^*},\dots,\sqrt{{\alpha_m}^*},\sqrt{\varepsilon_p}\right).$$
 	\end{enumerate}
 	We similarly prove the rest.	
 \end{proof}

 Now we shall construct the Hilbert genus field of $ K=\mathbb{Q}(\sqrt{-a{\varepsilon}_p \sqrt{p}})$, with $p=2$.
 Let us now assume that   $a=\displaystyle\prod_{i=1}^n q_i$ is an odd positive square-free integer such that 	$q_j\equiv \pm 1\pmod 8$ for all  $1\leq j \leq m$ $( m\leq n$, the case $m=0$ is possible here$)$. 
 
 \begin{theorem}\label{third main theorem}
 	Let $a$ be as above.  For any $j$ such that   $1\leq j \leq m$, let   $x_j$ and $y_j$ be the  two   positive integers such that $q_j=x_j^2-2y_j^2$.  Put $\alpha_j = x_j+y_j\sqrt{2} $, then  
 	the Hilbert genus field of $ K=\mathbb{Q}(\sqrt{-a{\varepsilon}_2 \sqrt{2}})$ is: 
 	$$E=K\left(\sqrt{{q_1}^*}, \sqrt{{q_2}^*}, \dots,\sqrt{{q_n}^*}, \sqrt{{\alpha_1}^*}, \sqrt{{\alpha_2}^*},\dots,\sqrt{{\alpha_m}^*}\right),$$
 	where  $\alpha_j^*$ is defined as follows:
 	
 	\begin{enumerate}[\rm $\bullet$]
 		\item If $q_j\equiv 1\pmod 8, \alpha_j^*:=(-1)^{\frac{x_j-1}{2}}\alpha_j,$  
 		\item	 If $q_j\equiv -1\pmod 8,$ then $x_j$ and $y_j$ are odd and we put  
 		$$\alpha_j^*:= \left\{\begin{array}{cll}
 			-\varepsilon_2\alpha_j,&	\text{ if }    x_j \equiv \pm1\pmod 4 \text{ and }    y_j\equiv -1\pmod 4, \\
 			\varepsilon_2\alpha_j,&\text{ else. }  
 			
 		\end{array}\right. $$
 	\end{enumerate}
 \end{theorem}
 \begin{proof}
 	It is known that for a prime number $q\equiv \pm 1\pmod 8$, there are positive integers $x, y$ such that  $q=x^2-2y^2$, and  if $q\equiv 1\pmod 8$, the integers $x$ and $y$ may be chosen such that  $x\equiv 1\pmod 2$ and $y\equiv 0\pmod 4$ (cf. \cite{leowillia}). 
 	
 	\noindent$\bullet$ Assume that $q_j\equiv 1\pmod 8$, for some  $1\leq j \leq m$. We have $(\alpha_j)=(x_j+y_j\sqrt{2})$ and the prime ideals of $k_0$ above $(q_j)$, are ramified in $K$. 
 	Then the ideal $(\alpha_j)$ is the square of a fractional ideal of $K$. By the definition of $\alpha_j^*$, we find $\alpha_j^*\equiv 1\pmod 4$ and the ideal ($\alpha_j^*$) is the square of   a fractional ideal of $K$ for $1\leq j\leq m$. So by Proposition \ref{prop 2.7}, $K(\sqrt{\alpha_j^*})/K$ is an unramified extension.
 	
 	\noindent$\bullet$ Assume that $q_j\equiv -1\pmod 8$, for some  $1\leq j \leq m$.	Let us firstly show that $y_j$ is odd. It is clear that $x_j$ is odd. We have,
 	$2y_j^2\equiv x^2-q  \equiv2\pmod 8$. By easy calculus one can check that the classes $\overline{y}$ of  $\mathbb{Z}/8\mathbb{Z}$ such that $2\overline{y}^2=\overline{2}$ are exactly the classes of odd integers $y$. Thus $y_j$ is odd. If $x_j\equiv y_j\equiv -1\pmod 4$, then we have
 	$ \alpha_j^*=- (1+\sqrt{2})(x_j+y_j\sqrt{2})= -(x_j+2y_j)-(x_j+y_j)\sqrt{2}$. Thus, 
 	$\alpha_j^* \equiv -1-2\sqrt{2}\equiv(3-2\sqrt{2})\equiv(1-\sqrt{2})^2 \pmod 4$. Since $q_j$ ramify in $K/k_0$, $(\alpha_j)$ is the square of an ideal of $K$. Therefore, by Proposition \ref{prop 2.7}, $K(\sqrt{\alpha_j^*})/K$ is an unramified extension. We similarly proceed for the other cases of $x_j$ and $y_j$.

 	Note that, by Lemma \ref{lemma 2.1; discriminants},  $\delta_{K/k_0}= 4a\sqrt{2}$. So  the   prime ideals of $k_0$ which ramify in $K$ are the prime divisors of $a$ in $k_0$ and the prime ideal $(\sqrt{2})$.
 	Therefore,  $r_2(\Delta/{K^*}^2)=n+m$. Hence, we show as in the proofs of the previous theorems that the set: $$\mathbb{B}=\{q_1^*, q_2^*, \dots, q_n^*, \alpha_1^*, \alpha_2^*, \dots, \alpha_m^*\}$$ is a representation of $\Delta /{K^*}^2$. Finally, we get:
 	$$E=K(\sqrt{{q_1}^*}, \sqrt{{q_2}^*}, \dots,\sqrt{{q_n}^*}, \sqrt{{\alpha_1}^*}, \sqrt{{\alpha_2}^*}, \dots, \sqrt{{\alpha_m}^*}),$$
 	which completes the proof.
 \end{proof}

 We close this section with some numerical examples.
 
 \section*{{Examples}}
 \begin{enumerate}[\indent\rm 1.]
 	\item Theorem \ref{first main theorem}: (the case  $a\equiv 3\pmod 4$). Let $K=\mathbb{Q}(\sqrt{-42427(5+2\sqrt{5})}\,)=\mathbb{Q}(\sqrt{-42427(2+\sqrt{5})\sqrt{5}}\,)$. We have $p=5\equiv 5 \pmod 8$, $\varepsilon_5=2+\sqrt{5}$ is the fundamental unit of $\mathbb{Q}(\sqrt{5})$, $42427=7\times 11\times 19\times 29\equiv 3\pmod 4$ and $(\frac{5}{7})=-1$, $(\frac{5}{11})$ = $(\frac{5}{19})$ = $(\frac{5}{29})$ = $1$, then $n=4$ and $m=3$. Thus,  $r_2(\Delta/{K^*}^2)$ = $n+m+1$ = 8. As $11=4^2-5$, $19=8^2-5\times 3^2$ and $29=7^2-5\times 2^2$, it follows that  $ \alpha_1=4+\sqrt{5}, \alpha_2=8+3\sqrt{5}$ and $\alpha_3=7+2\sqrt{5}$,
 	so $\alpha_1^*=5+4\sqrt{5}$, $\alpha_2^*=-(15+8\sqrt{5})$ and $\alpha_3^*=7+2\sqrt{5}$.\\
 	Hence:
 	$$E= K(\sqrt{-3}, \sqrt{-7}, \sqrt{-11}, \sqrt{-19}, \sqrt{29}, \sqrt{\alpha_1^*}, \sqrt{\alpha_2^*}, \sqrt{\alpha_3^*})$$
 	is the Hilbert genus fields of $K$.\\
 	\item Theorem \ref{second main theorem}: (the case $a\equiv 1\pmod 4$). Let $K=\mathbb{Q}(\sqrt{-4199\varepsilon_{73}\sqrt{73}}\,)$, we have: $73\equiv 1\pmod 8$, $4199=13\times 17\times 19\equiv 1 \pmod 4$,  $(\frac{73}{13})$ = $(\frac{73}{17})$ =$-1$ and $(\frac{73}{19})$ = 1. Thus, $n=3,  m=1$.  Therefore, $r_2(\Delta/{K^*}^2)$ = $n+m+2$=6.\\
 	As $19=26^2+3\sqrt{73}$, then $\alpha_1=26+3\sqrt{73}$ and $ \alpha_1^*=219+26\sqrt{73}$. Hence,   the Hilbert genus fields of $K$ is:
 	$$E= K(\sqrt{-1}, \sqrt{13}, \sqrt{-17}, \sqrt{-19}, \sqrt{\varepsilon_{73}}, \sqrt{219+26\sqrt{73}}),$$
 	where $\varepsilon_{73}$ is the fundamental unit of $\mathbb{Q}(\sqrt{73})$.
 	\item Theorem \ref{third main theorem}: (the case $p=2$). Let $K=\mathbb{Q}(\sqrt{595(1+\sqrt{2})\sqrt{2}}\,)$. We have $595= 7\times 17\times 5$, $(\frac{2}{5})=-1$ and $(\frac{2}{7}) = (\frac{2}{17}) =  1$. Then $r_2(\Delta/{K^*}^2) = n+m = 3+2 = 5$, by $7\equiv 7\pmod 8$, $17\equiv 1\pmod 8$, $7=3^2-2\times 1^2$ and $17 = 7^2-2\times 4^2$. So $\alpha_1 = 3+1\sqrt{2}, \alpha_2 = 7+4\sqrt{2}$,   $\alpha_1^* = (1+\sqrt{2})(3+\sqrt{2})=5+4\sqrt{2}$ and $\alpha_2^* = -(7+4\sqrt{2})$. Hence, the Hilbert genus field of $K$ is:
 	$$E= K(\sqrt{5}, \sqrt{-7}, \sqrt{17}, \sqrt{5+4\sqrt{2}}, \sqrt{-(7+4\sqrt{2})}).$$.
 \end{enumerate}

 \section*{Acknowledgment}
 The authors are so grateful to Professor Mohammed Taous for  reading the preliminary versions of this paper as well for many helpful suggestions and discussions.

\end{document}